\documentclass[11pt]{amsart}

\usepackage{color,graphicx,enumerate,wrapfig,amssymb}
\usepackage{subfigure}
\usepackage{eucal}

\usepackage{setspace}

\usepackage{hyperref}

\hypersetup{
    bookmarks=true,         
    unicode=true,          
    pdftoolbar=true,        
    pdfmenubar=true,        
    pdffitwindow=false,     
    pdfstartview={FitH},    
    pdftitle={My title},    
    pdfauthor={Author},     
    pdfsubject={Subject},   
    pdfcreator={Creator},   
    pdfproducer={Producer}, 
    pdfkeywords={keyword1} {key2} {key3}, 
    pdfnewwindow=true,      
    colorlinks=true,       
    linkcolor=blue,          
    citecolor=blue,        
    filecolor=magenta,      
    urlcolor=cyan           
}

\renewenvironment{proof}[1][\proofname ]{{\noindent \bfseries #1. }}{\qed \bigskip }

\newcommand{\R}{{\mathbb R}}
\newcommand{\Z}{{\mathbb Z}}

\newcommand{\e}{\varepsilon}

\newtheorem{theorem}{Theorem}[section]

\newtheorem{lem}[theorem]{Lemma}

\newtheorem{prop}[theorem]{Proposition}
\newtheorem{remark}[theorem]{Remark}

\numberwithin{equation}{section}

\title[Oscillatory integrals operators]{$L^2$-estimates for singular oscillatory integral operators}

\author{Hayk Aleksanyan}
\address{School of Mathematics, The University of Edinburgh, JCMB The King's Buildings, Peter Guthrie Tait Road, Edinburgh EH9 3FD}
\email{hayk.aleksanyan@gmail.com}

\author[Henrik Shahgholian ]{Henrik Shahgholian}
\address{Department of Mathematics, KTH Royal Institute of Technology,
  100~44  Stockholm, Sweden}
\email{henriksh@kth.se}

\author{Per Sj\"{o}lin}

\address{Department of Mathematics, KTH Royal Institute of Technology,
  100~44  Stockholm, Sweden}
\email{persj@kth.se}

\begin{document}
    \begin{abstract}

      In this note we study singular oscillatory integrals with linear phase function over hypersurfaces which may oscillate,
      and prove estimates of $L^2 \mapsto L^2$ type for the operator, as well as for the corresponding maximal function.
      If the hypersurface is flat, we consider a particular class of a nonlinear phase functions, and apply
      our analysis to the eigenvalue problem associated with the Helmholtz equation in $\mathbb{R}^3$.


\vspace{0.2cm}

    \end{abstract}

\maketitle


\section{Introduction}\label{sec-intro}

In their seminal work, D.H. Phong and E.M. Stein \cite{Phong-Stein} study new classes of oscillatory integral operators
with singular weights, that in turn apply to the study of certain PDE problems.
Since this work appeared, there have been numerous  results 
and developments of the theory, with large variety of these type of operators.
In this note we will be interested in two particular aspects of the result of Phong-Stein.
Apart from \cite{Phong-Stein}, our interest is
partially motivated by applications of oscillatory integral operators in the anaysis of boundary value problems 
(see for example \cite{ASS1}-\cite{ASS3}), and recent developments in PDE problems involving oscillating boundaries (see for example \cite{KP}).
In the first part of the note we will introduce and study oscillatory integral operators
with singular kernel and linear phase, where integration is carried out on
smooth hypersurfaces. Here we will aim at obtaining precise estimates with respect to the smoothness norm of the
hypersurface. This is done in order to obtain non trivial bounds when one allows 
the hypersurface to oscillate. A type of an oscillating hypersurface considered here
is technically smooth, however due to its oscillatory nature one can not rely in a straightforward manner on partial integration techniques
to bound the corresponding integral operator,
since the derivatives of the graph representing the surface blow-up. 
In the second part of the paper, we will study a similar problem but with nonlinear phase function.
The type of phase function studied in this case is in part motivated by the Helmholtz equation,
and we will discuss one particular application of our results regarding decay estimates
for the eigenvalue problem corresponding to the Helmholtz equation in $\R^3$.

\vspace{0.2cm}
Let us start by fixing the setup and notation.
Throughout the text, by $``\cdot"$ we denote the standard scalar product in $\R^n$.
For two quantities $x$ and $y$ we write $x \lesssim y$ if there is an absolute constant $C$
for which $x\leq C y$. Likewise, if $x$ and $y$ depend on some parameter, say $\delta$, we may write
$x \lesssim_\delta y$ to indicate that the constant in the inequality depends on $\delta$,
and is otherwise absolute.

\vspace{0.2cm}

For a real-valued function $\psi\in C^\infty(\R^n)$ satisfying $|| D^\alpha \psi||_{L^\infty(\R^n)} \lesssim_\alpha 1$
for any $\alpha \in \Z^n_+$, let $\Gamma$ be the following hypersurface 
\begin{equation}\label{Gamma-surface}
\Gamma=\{ (y,\psi(y))\in \R^{n+1}: \ y\in \R^n \}.
\end{equation}
For $f\in L^2(\R^n)$, $\lambda>0$, and $(x,x_{n+1})\in \R^n\times \R$ define
\begin{equation}\label{def-T-lambda}
 T_\lambda  f(x, x_{n+1}) = \int\limits_{\Gamma } e^{ i\lambda x \cdot y   }
 \varphi_0( (x,x_{n+1}) ,y ) K(x-y, x_{n+1} - y_{n+1}) f(y) d \sigma (y,y_{n+1}),
\end{equation}
where  $ d\sigma $ is the induced surface measure on $\Gamma$, and

\begin{doublespace}
\begin{itemize}
 \item[(A1)]\label{(A1)} $ \varphi_0 $ is real-valued and from the class $C_0^\infty ( \R^{n+1} \times \R^n )$,
 \item[(A2)]\label{(A2)} $K\in C^\infty ( \R^{n+1} \setminus \{0\} )$ and for any $z\in \R^{n+1}\setminus \{0\}$ and any $\alpha \in \Z^{n+1}_+$ we have
 $$ | D^\alpha K(z) | \lesssim_\alpha \frac{|z|^m}{ |z|^{n + |\alpha|} }, $$ 
 where $0\leq m<n$. Here $n\geq 1$, and we do not assume that $m$ is necessarily an integer.
\end{itemize}
\end{doublespace}

\noindent \textbf{Organization.} In Section \ref{sec-est-of-T-lambda} we study 
$T_\lambda$ as an operator from $L^2( \R^{n} )$ to $L^2(\R^{n} )$ and prove decay estimates for its operator norm as $\lambda \to \infty$.
A special attention is paid to obtaining precise bounds with respect to the smoothness norms of the surface $\Gamma$.
We then conclude the section by discussing the behavior of the operator $T_\lambda$ under perturbations of the fixed surface $\Gamma$.
Next, in Section \ref{sec-maximal-operator} we consider a maximal operator associated with operators of the form (\ref{def-T-lambda})
when the surface is allowed to oscillate. More precisely, for a
family of surfaces $\{\Gamma_\e\}_{0<\e \leq 1}$,
we let $T_\lambda^\e$ be the operator defined by (\ref{def-T-lambda}) for the surface $\Gamma_\e$, where the parameter $0<\e\leq 1$ is meant to
model an oscillatory behavior of the given family of hypersurfaces.
Then we analyse boundedness of the following maximal operator $T_\lambda^*f (x,x_{n+1}) =\sup\limits_{0<\e\leq 1} |T_\lambda^\e f (x,x_{n+1})|$,
where $(x,x_{n+1})\in \R^n \times \R$.
Finally, Section \ref{sec-nonlinear} studies operators of type $T_\lambda$, however instead of a linear phase, there we have
a ``fractional"-type nonlinearity, namely $x\cdot y$ in the exponential is replaced by $|x-y|^\gamma$ with $\gamma\geq 1$. This change requires
a radically different approach. We discuss in subsection \ref{sub-sec-Helmholtz} how the case $\gamma=1$ applies to Helmholtz
equation. At the end of Section \ref{sec-nonlinear} we 
show that the obtained upper bounds of some of the operators considered in the article are sharp.

\section{The main estimate for $T_\lambda$}\label{sec-est-of-T-lambda}

Throughout this section we will be working with operators defined on $L^2(\R^n)$
and with values in $L^2(\R^n)$. Thus if $T$ is an operator of this type,
by $||T|| $ we denote its operator norm. 
Also, when estimating a certain quantity, we will be only concerned
with constants that depend on the surface, i.e. the function $\psi$ of (\ref{Gamma-surface}).
The following is our main estimate for the operator $T_\lambda$ defined by (\ref{def-T-lambda}).

\vspace{0.2cm}

\begin{theorem}\label{Thm-main-on-T-lambda}
For any $1\leq m< n$, and any $x_{n+1}\in \R$ we have
\begin{equation}
 || T_\lambda||   \lesssim \lambda^{-\frac{m}{2} \frac{n }{n+1}}  [ 1+ || \nabla  \psi||_{ L^\infty(\R^n ) }  ] \times
 [ 1+ || \psi||_{ L^\infty(\R^n) }   ],
\end{equation}
where the constant depends on the cut-off function $\varphi_0$, and is independent of $x_{n+1}$ and $\psi$.
\end{theorem}

\vspace{0.2cm}

We start by decomposing $T_\lambda$ as follows. Fix a parameter $0\leq \beta \leq 1$ that will be specified below,
and a smooth function $\varphi :\R^n \to [0,1] $ such that $ \varphi(y)=0 $ for $|y|\leq 1$
and $\varphi(y)=1$ for $|y|\geq 2$.
We then have the following decomposition
$$
T_\lambda f  = A_\lambda f  + B_\lambda  f,
$$
where $B_\lambda$ in terms of volume integral is
\begin{multline*}
B_\lambda f(x,x_{n+1}) = \int_{\R^n } e^{ i \lambda x \cdot y   } 
\varphi_0( (x,x_{n+1}) , y ) K(x -y , x_{n+1} -   \psi(y  ) )
\times  \\  \varphi((x -y ) \lambda^\beta )  ( 1+|   \nabla  \psi( y   ) |^2  )^{1/2}   f(y ) d y .
\end{multline*}

\vspace{0.2cm}

\begin{lem}\label{Lem-A-lambda}
For any $\lambda \geq 1$, any $m>0$, and any $0 \leq  \beta \leq 1$ one has $$
|| A_\lambda||  \lesssim [ 1+ || \nabla  \psi ||_{L^\infty (\R^n ) } ]  \lambda^{-m \beta}.
$$
\end{lem}  
\begin{proof}
Rewriting $A_\lambda $ in terms of volume integral
we get
\begin{multline*}
A_\lambda f(x, x_{n+1}) = \int_{\R^n } e^{ i \lambda x \cdot y   }  
\varphi_0((x,x_{n+1}) ,y ) K(x -y , x_{n+1} -  \psi(y  ) )
\times  \\ [ 1-\varphi((x -y ) \lambda^\beta )] ( 1+| \nabla  \psi( y  ) |^2  )^{1/2}   f(y  ) d y .
\end{multline*}

Now, using assumption (A2) for kernel $K$, along with the definition of $\varphi$ we obtain
\begin{equation}\label{A-est1}
 | A_\lambda f(x,x_{n+1}) | \lesssim [ 1+|| \nabla  \psi ||_{L^\infty (\R^n) } ] \mathcal{I}_\lambda (x ),
\end{equation}
where the constant depends on the supremum norm of $\varphi_0$, and the appropriate constant involved in (A2), and we have
\begin{multline*}
\mathcal{I}_\lambda (x ): = \int\limits_{ | x  -y  |\leq 2\lambda^{-\beta} }  \frac{| f(y ) | }{  | x  - y  |^{n  - m } } dy  \leq \\ 
\sum\limits_{2^{-k} \leq 4\lambda^{-\beta}} \int\limits_{ 2^{-k-1} \leq |x-y| \leq 2^{-k} } \frac{| f(y) | }{  | x - y |^{n  -m } } dy  \lesssim \\
\sum_{ 2^{-k} \leq 4 \lambda^{-\beta} }  2^{ -m k } \frac{1}{2^{-k n}} \int\limits_{ 2^{-k-1} \leq |x -y | \leq 2^{-k} } | f(y ) |    dy  \lesssim \\ 
\mathcal{M} f(x ) \sum_{ 2^{-k} \leq 4 \lambda^{-\beta} }  2^{ -m k } \lesssim 
\mathcal{M} f(x ) \lambda^{-\beta m},
\end{multline*}
where $\mathcal{M}$ is the Hardy-Littlewood maximal function.
Since $\mathcal{M}$ has strong $(2,2)$ type, the last inequality combined with (\ref{A-est1})
completes the proof of the Lemma.
\end{proof}

\noindent To study $B_\lambda f $ we rewrite it in the form
\begin{equation}\label{def-B-lambda}
B_\lambda f(x,x_{n+1}) = \int_{\R^n } k_\lambda \big( (x,x_{n+1}), y \big) K \big(x - y , x_{n+1} -\psi(y ) \big) g(y ) dy ,
\end{equation}
where we have set
$$
k_\lambda \big( (x,x_{n+1}), y \big) := e^{ i \lambda x  \cdot y  } \varphi_0 \big( (x,x_{n+1}), y \big) \varphi( (x -y ) \lambda^{\beta} ) , 
$$
and
\begin{equation}\label{g-in-B-lambda}
g(y ) :=   f(y ) (  1+|\nabla \psi(y )|^2 )^{1/2} , \qquad  y \in \R^n .
\end{equation}
For the estimate of $B_\lambda$ we start with a lemma, where the constants are allowed to depend on the norm
of derivatives of the function $\psi$ representing the surface. This dependence will be revised later on.

\begin{lem}\label{Lem-fixed-surface}
For $(x,x_{n+1}) \in \R^n \times \R $ set
$$
(\widetilde{B}_\lambda f)(x,x_{n+1} ) = \int_{\R^n } k_\lambda \big( (x,x_{n+1}) , y  \big) K \big(x  - y , x_{n+1}- \psi(y )\big) f(y ) dy .
$$
Then
$$
|| \widetilde{B}_\lambda ||  \leq C_\psi \lambda^{ n (\beta - 1/2) - m\beta },
$$
where $0\leq m < n$, $C_\psi$ is independent of $\lambda$, and $x_n$, and depends on finite number
of derivatives of $\psi$.
\end{lem}

\begin{proof}
The proof follows closely the lines of Proposition 1 on p. 134 of \cite{Phong-Stein}.
Observe that $\widetilde{B}_\lambda$ has kernel 
$$
L_\lambda (x , y ) = e^{ i \lambda x  \cdot y   } \varphi_0( (x,x_{n+1}), y  ) \varphi( (x  -y ) \lambda^\beta ) K(x  - y , x_{n+1} - \psi(y ) ),
$$
and the conjugate operator $\widetilde{B}^*_\lambda$ has kernel $ \overline{L_\lambda (y  , x )} $. It follows that
$ \widetilde{B}_\lambda  \widetilde{B}^*_\lambda $ has kernel
\begin{multline*}
M_\lambda (x, y) = L_\lambda (x, z) \overline{L_\lambda ( y, z )} d z  = \\
\int_{\R^n } e^{ i\lambda ( x  - y  ) \cdot z  } \varphi_0( (x,x_{n+1}), z ) \varphi_0( (y, x_{n+1}) , z ) 
\varphi ( (x -z ) \lambda^\beta ) \varphi ( (y -z ) \lambda^\beta ) \times 
\\ K( x -z , x_{n+1}- \psi(z ) ) \overline{ K ( y - z , x_{n+1}- \psi(z ) )} dz .
\end{multline*}
Since we will obtain uniform estimates in $x_{n+1}$ the dependence of the kernels on $x_{n+1}$ is dropped from the notation.
Using the estimate (A2) for $K$, for any $\alpha \in \Z^n_+$ we have
$$
| (D_{z }^\alpha  K ) (x  - z ,  x_{n+1} - \psi(z ) )  | \lesssim_\psi \frac{1}{ | x - z |^{n  - m + |\alpha|} }.
$$
In view of the last inequality, integrating by parts $N\geq 0$ times  in the kernel $M$ we get
\begin{equation}\label{M-estimate}
| M_\lambda (x,y) | \lesssim_\psi ( \lambda |x - y| )^{-N} \sum\limits_{ 0\leq k+l \leq N } \mathcal{I}_{k,l}(x , y ),
\end{equation}
where
$$
 \mathcal{I}_{k,l} (x, y): = \int\limits_{\substack{ \lambda^{-\beta} \leq | x  - z  | \leq C \\  \lambda^{-\beta} \leq | y  - z  | \leq C  } } 
| x  -z  |^{-n+m - k}  | y - z   |^{-n+m  - l} d z .
$$
Assuming $n-2m+N>0$ and using $k+l\leq N$, it is easy to see that each $\mathcal{I}_{k,l}$ is uniformly bounded above by
$$
\int\limits_{ \lambda^{-\beta} \leq  |z  |\leq C } |z |^{ -2n+2m -N } dz  
\lesssim \int\limits_{ \lambda^{-\beta} }^1 \frac{r^{n-1}}{ r^{2n-2m +N} } dr \lesssim
\lambda^{ \beta (   n-2m+N  ) }.
$$
From here, getting back to (\ref{M-estimate}) we obtain
\begin{equation}\label{M-est-final}
| M(x  , y) |  \lesssim_\psi \frac{\lambda^{ \beta(n-2m ) }  }{ ( \lambda^{1-\beta} | x  - y | )^N  } ,
\end{equation}
where $N\geq 0$ is an integer satisfying $n-2m+N>0$. By a convexity argument as in Proposition 1 of \cite{Phong-Stein} p.134, we can eliminate
the condition on $N$ being an integer, thus getting (\ref{M-est-final}) with non integer values of $N$.
We then choose $N_1< n $ and $N_2>n $ in (\ref{M-est-final}),
which is allowable since $m<n$. It follows that
\begin{multline*}
\int_{ \R^n } | M(x , y ) | d y  \lesssim_\psi 
\int\limits_{ \lambda^{1-\beta} | z  | \leq 1  } \frac{\lambda^{ \beta(n-2m  ) }  }{ ( \lambda^{1-\beta} | z | )^{N_1}  } d z  +
\int\limits_{ \lambda^{1-\beta} | z  | > 1  } \frac{\lambda^{ \beta(n-2m  ) }  }{ ( \lambda^{1-\beta} | z | )^{N_2}  } d z  \lesssim \\
\big[ \text{substituting } u= \lambda^{1-\beta} z  \big] \ \lesssim_\psi  \lambda^{ n( 2\beta -1 ) - 2m \beta  },
\end{multline*}
uniformly in $x$. By symmetry, the same estimate holds for integration with respect to $x$ and fixed $y$.
Using the obtained mixed $L^1$-norm estimate of the kernel $M$ we can easily bound the $L^2 \mapsto L^2$ norm of the corresponding operator,
thus obtaining
$$
|| \widetilde{B}_\lambda || \lesssim_\psi \lambda^{ n (\beta - 1/2) - m\beta  },
$$
where we choose $0\leq \beta \leq 1$ satisfying 
$ \beta \leq  \frac{n }{ 2(n-m ) }$ to get a bound for the norm of $\widetilde{B}_\lambda$. The proof is complete.
\end{proof}

We are now ready to prove the main estimate on $B_\lambda$ defined by (\ref{def-B-lambda}).

\begin{prop}\label{Thm-B-lambda-0}
Assume $1\leq m <n $, and for fixed $(x, x_{n+1}) \in \R^n\times \R$ let $B^{(0)}_\lambda$ be defined as follows
$$
B_\lambda^{(0)} g(x,x_{n+1}) = \int_{\R^n } e^{ i \lambda x  \cdot y  } \varphi_0( (x,x_{n+1}) , y ) 
\varphi( (x  - y ) \lambda^\beta ) K(x  - y , x_{n+1} - \psi(y ) ) g(y ) d y  .
$$
Then
$$
|| B_\lambda^{(0)} ||  \lesssim   
\lambda^{ n (\beta -1/2) -m \beta  } + || \psi||_{L^\infty} \lambda^{ n (\beta -1/2) -m\beta+\beta  }.
$$
\end{prop}

\begin{remark}
Observe that if $B_\lambda$ is the operator defined by (\ref{def-B-lambda}), then
by (\ref{def-B-lambda}) and (\ref{g-in-B-lambda}) we have
$ B_\lambda f = B_\lambda^{(0)} g $ with $g(y ) = f(y )   ( 1+|\nabla  \psi(y )|^2 )^{1/2}$.
Hence Proposition \ref{Thm-B-lambda-0} implies the following bound
\begin{equation}\label{B-lambda-norm}
|| B_\lambda ||  \lesssim [ 1+|| \nabla  \psi||_{L^\infty{ (\R^n )} } ] \times 
\left( \lambda^{ n (\beta -1/2) -m \beta   } + || \psi||_{L^\infty(\R^n) } \lambda^{ n (\beta -1/2) -m\beta+ \beta  } \right).
\end{equation}

\end{remark}

\bigskip

\begin{proof}[Proof of Proposition \ref{Thm-B-lambda-0}]
We have
$$
K \big(x - y , x_{n+1} - \psi(y ) \big) - K(x  - y , x_{n+1}) = \int\limits_{ x_{n+1} }^{ x_{n+1} - \psi (y ) } \partial_{n+1} K(x  - y , t) dt.
$$
Recall that $ k_\lambda ( (x,x_{n+1}), y ) = e^{ i \lambda x  \cdot y  } \varphi_0( (x,x_{n+1}), y ) \varphi( (x  - y ) \lambda^\beta ) $,
we thus get
\begin{multline*}
B_\lambda^{(0)} g(x,x_{n+1} ) = \int_{\R^n } k_\lambda \big( (x,x_{n+1}), y \big) K \big(x  - y , x_{n+1} - \psi(y ) \big) g(y ) d y  = \\
\int_{ \R^n } k_\lambda \big((x,x_{n+1}), y \big) K\big(x  - y  , x_{n+1} \big) g(y ) d y   + \\
\int_{ \R^n } k_\lambda \big( (x,x_{n+1}), y \big) \left( \int\limits_{ x_{n+1} }^{ x_{n+1} - \psi(y) } 
\partial_{n+1} K(x  - y  , t) dt  \right) g(y ) d y  =: \\ B^{(1)}_\lambda g(x,x_{n+1}) + B^{(2)}_\lambda g(x,x_{n+1}). 
\end{multline*}
It follows by Lemma \ref{Lem-fixed-surface} that
\begin{equation}\label{est-on-D-lambda-fixed}
|| B^{(1)}_\lambda  ||  \lesssim \lambda^{ n( \beta - 1/2) - m \beta    }.
\end{equation}

\noindent For $B^{(2)}_\lambda$ we have the following decomposition
\begin{multline}
B^{(2)}_\lambda g(x,x_{n+1}) = - \iint\limits_{E_1} k_\lambda ( (x,x_{n+1}), y ) \partial_{n+1} K(x  - y , t) g(y ) d y  dt + \\
\iint\limits_{E_2} k_\lambda ( (x,x_{n+1}) ,y ) \partial_{n+1} K(x  - y , t) g(y ) d y  dt,
\end{multline}
where for $M=|| \psi ||_{L^\infty(\R^n ) }$ we let
\begin{multline*}
E_1 := \{ (y , t): \ \psi(y ) >0, \ x_{n+1} - \psi(y ) < t< x_{n+1} \} = \\ \{ (y , t): \ x_{n+1}  - M <t <x_{n+1}, \ \psi(y ) >x_{n+1} -t \}
\end{multline*}
and
\begin{multline*}
E_2 := \{ (y, t): \ \psi(y ) < 0, \ x_{n+1} < t <  x_{n+1}+ \psi(y )  \} =\\ \{ (y , t): \ x_{n+1}   <t < x_{n+1} +M , \ \psi(y ) <x_{n+1} -t \}.
\end{multline*}
From here we obtain  
\begin{multline}
 B_\lambda^{(2)} g(x, x_{n+1}) = \\
 - \int\limits_{x_{n+1} - M}^{x_{n+1}} \int\limits_{\R^n} k_\lambda ( (x,x_{n+1}), y) 
 \partial_{n+1} K(x  - y , t) \mathbb{I}_{ \{ y : \ \psi(y ) > x_{n+1} -t \} } g(y ) dy dt + \\
 \int\limits_{x_{n+1} }^{x_{n+1} + M} \int\limits_{\R^n}  
 k_\lambda ( (x,x_{n+1}), y ) \partial_{n+1} K(x  - y , t) \mathbb{I}_{ \{ y : \ \psi(y ) < x_{n+1} -t \} } g(y ) dy dt,
\end{multline}
where $\mathbb{I}$ stands for the characteristic function. Denote
$ g_{1,t} (y) : = \mathbb{I}_{ \{ y : \ \psi(y ) > x_{n+1} -t \} } g(y ) $,
$g_{2,t}(y ):= \mathbb{I}_{ \{ y : \ \psi(y ) < x_{n+1} -t \} } g(y )$, and for $h\in L^2(\R^n)$ set
$$
F_t h(x,x_{n+1}) : = \int_{\R^n } k_\lambda ( (x,x_{n+1}),y ) \partial_{n+1} K(x  - y  , t) h(y ) d y , \qquad  (x ,x_{n+1}) \in \R^{n+1}. 
$$
With this notation we have
$$
B_\lambda^{(2)} g(x,x_{n+1}) = -\int\limits_{x_{n+1} - M }^{x_{n+1}} F_t g_{1,t} (x,x_{n+1}) dt + 
\int\limits_{x_{n+1}  }^{x_{n+1} + M} F_t g_{2,t} (x,x_{n+1}) dt.
$$
Observe that Lemma \ref{Lem-fixed-surface} implies $ || F_t||  \lesssim \lambda^{ n (\beta -1/2) -m\beta + \beta  } $,
with constants independent of $\psi$, hence using Minkowski's inequality for integrals we obtain
\begin{equation}\label{est-D-lambda-2}
|| B_\lambda^{(2)} ||  \lesssim M \lambda^{ n (\beta -1/2) -m\beta+ \beta  }.
\end{equation}
Finally, combining (\ref{est-on-D-lambda-fixed}) and (\ref{est-D-lambda-2}) we get
$$
||B_\lambda^{(0)}||  \lesssim \lambda^{ n (\beta -1/2) -m \beta     } +
|| \psi||_{L^\infty(\R^n ) } \lambda^{ n (\beta -1/2) -m\beta+ \beta  },
$$
completing the proof the Proposition.
\end{proof}

\begin{proof}[Proof of Theorem \ref{Thm-main-on-T-lambda}]
Getting back to the original operator $T_\lambda$ defined by (\ref{def-T-lambda}),
in view of Lemma \ref{Lem-A-lambda} and Proposition \ref{Thm-B-lambda-0} we have
\begin{multline*}
 || T_\lambda || \leq || A_\lambda||   + || B_\lambda||  \lesssim \\ [ 1+ || \nabla \psi||_{ L^\infty(\R^n ) }  ] \times  [ \lambda^{-m\beta} +  
  \lambda^{ n(\beta - 1/2 ) - m\beta    }   + || \psi||_{ L^\infty(\R^n)} \lambda^{ n ( \beta -  1/2 ) - m\beta +  \beta }  ],
\end{multline*}
which, if optimized in $\beta$ implies
\begin{equation}
 || T_\lambda||   \lesssim \lambda^{-\frac{m}{2} \frac{n }{n+1}}  [ 1+ || \nabla  \psi||_{ L^\infty(\R^n ) }  ] \times
 [ 1+ || \psi||_{ L^\infty(\R^n) }   ].
\end{equation}

The last estimate completes the proof of Theorem.
\end{proof}

\subsection{Perturbing the surface}\label{sec-perturb-surface}
Here we discuss the case when the operator $T_\lambda$ is defined on a perturbation of some fixed surface $\Gamma_0$.
Assume we have 
$$
\Gamma_0 = \{ (y,\psi_0(y)): \ y\in \R^n  \},
$$
where $\psi_0\in C^\infty(\R^n)$ is real-valued and has all its derivatives bounded on $\R^n$.
Now, for $\e>0$ consider a family of smooth and real-valued functions $\psi_\e(y)$, along with the hypersurfaces
$$
\Gamma_\e = \{ ( y,\psi_0(y) +\psi_\e(y) ) : \ y\in \R^n  \},
$$
that is we perturb the fixed surface $\Gamma_0$ by $\psi_\e$. Here as well, the assumption is that each $\psi_\e$
has bounded derivatives of all orders on $\R^n$, however we do not impose any uniform bound with respect to $\e$,
neither we assume any structural restriction on $\psi_\e$.

For $\e>0$ let the operator $T_\lambda^\e$ be defined as in (\ref{def-T-lambda}) where integration is over
$\Gamma_\e$. Then by Theorem \ref{Thm-main-on-T-lambda} we get
$$
||T_\lambda^\e||  \lesssim_{\psi_0} \lambda^{-\frac{m}{2} \frac{n }{n+1}}  [ 1+ || \nabla  \psi_\e  ||_{ L^\infty(\R^n ) }  ] \times
 [ 1+ || \psi_\e ||_{ L^\infty(\R^n) }   ],
$$
uniformly in $\e>0$ and $\lambda>0$. The point of the last estimate is that even with a rough surface,
in a sense that there is no uniform control over the derivatives of the graph representing the surface, we may still control
the norms of the corresponding operators.


\section{Maximal operator for oscillating surfaces}\label{sec-maximal-operator}

The aim of this section is to illustrate that for small oscillations of the surface,
we may as well control the maximal operator associated with the oscillations. For $\e>0$ set
\begin{equation}\label{Gamma-eps}
 \Gamma_\e = \{ ( y , \e^{\gamma} \psi( y  / \e ) ): \ y \in \R^n \},
\end{equation}
where $ \psi \in C^\infty (\R^n) $ is bounded and has bounded derivatives
of all orders, and $\gamma>0$ is a fixed parameter which will be specified below in Theorem \ref{Thm-bdd-maximal-op}.
As in (\ref{def-T-lambda}), for $\lambda>0$, $\e >0$ and $(x,x_{n+1}) \in \R^{n}\times \R$ define
\begin{equation}
 T_\lambda^\e f(x,x_{n+1}) = \int\limits_{\Gamma_\e} e^{ i\lambda x \cdot y } 
 \varphi_0 \big( (x,x_{n+1}), y \big) K(x-y, x_{n+1} - y_{n+1}  ) f(y ) d \sigma_\e (y,y_{n+1}), 
\end{equation}
where $f\in L^2(\R^n)$ and $\varphi_0$ and $K$ are defined as in Section \ref{sec-intro}.
Consider the following maximal operator
\begin{equation}\label{max-oper-T}
T_\lambda^* f(x,x_{n+1}) : = \sup\limits_{ 0<\e \leq 1 } | T_\lambda^\e f(x,x_{n+1}) |, \qquad (x,x_{n+1}) \in \R^n\times \R.
\end{equation}
Similar type of maximal operators related to integral operators
for a parameterized family of smooth surfaces had been considered in \cite{Sogge-Stein}.
More precisely \cite{Sogge-Stein} deals with maximal operator corresponding to $L^p$-averaging operators defined on a family of surfaces converging to
a smooth immersed surface with Gaussian curvature nowhere vanishing of infinite order.
While the idea of taking the pointwise upper bound with respect to the family of hypersurfaces is the same as in \cite{Sogge-Stein},
the problem and the analysis discussed here  are completely different from \cite{Sogge-Stein}.
Our main result concerning (\ref{max-oper-T}) is the following $L^2(\R^n) \mapsto L^2(\R^n) $ bound.

\begin{theorem}\label{Thm-bdd-maximal-op}
For $\gamma> 3/2 $ and $1\leq m<n$, we have $ || T_\lambda^* ||  \lesssim \lambda^{  -\frac{ m}{2} \frac{n}{n+2}    } $.
\end{theorem}

\bigskip

As for $T_\lambda$ defined by (\ref{def-T-lambda}), start
by fixing a smooth function $\varphi :\R^n \to [0,1] $ such that $ \varphi(y )=0 $ for $|y |\leq 1$
and $\varphi(y )=1$ for $|y |\geq 2$. Fix also a parameter $0<\beta<1$ that will be specified in the proof of Theorem \ref{Thm-bdd-maximal-op} below.
We have the following decomposition
$$
T_\lambda^\e f = A_\lambda^\e f + B_\lambda^\e f ,
$$
where 
$$
A_\lambda^\e f(x,x_{n+1}) = \int\limits_{\Gamma_\e} e^{ i\lambda x \cdot y } 
\varphi_0( (x,x_{n+1}) , y ) K(x-y,x_{n+1} - y_{n+1}  ) [ 1-\varphi((x -y ) \lambda^\beta )]  f( y ) d \sigma_\e(y).
$$
Clearly we get
$$
T_\lambda^* f(x,x_{n+1}) \leq A^*_\lambda f(x,x_{n+1}) + B_\lambda^* f(x,x_{n+1}),
$$
where the maximal operators on the right-hand side are defined in analogy with (\ref{max-oper-T}).

\begin{lem}\label{Lem-A-lambda-max}
For any $\lambda \geq 1$, any $\gamma \geq 1$, and any $ \beta \geq 0 $ one has
$$
|| A_\lambda^* ||  \lesssim  [ 1+ || \nabla  \psi ||_{L^\infty(\R^n )} ]   \lambda^{-m \beta}.
$$
\end{lem}

\begin{proof}
We use the fact that $\gamma\geq 1$ and repeat the proof of Lemma \ref{Lem-A-lambda}.
\end{proof}

\begin{proof}[Proof of Theorem \ref{Thm-bdd-maximal-op}]
For fixed $\e>0$ and $(x,x_{n+1}) \in \R^n\times \R$ we have
\begin{equation}\label{B-integral}
B_\lambda^\e f(x,x_{n+1}) = B_\lambda^0 f(x,x_{n+1}) + \int\limits_0^\e \frac{d}{d \tau} B_\lambda^\tau f(x,x_{n+1})   d\tau,
\end{equation}
where $B_\lambda^0 f(x,x_{n+1}) = \lim\limits_{\e \to 0} B_\lambda^\e f(x,x_{n+1}) $. 
Since $\gamma>1$ this limit exists, as well as the differential with respect to $\tau$ in (\ref{B-integral}).
By (\ref{B-integral}) we have
$$
|B_\lambda^* f | \leq  | B_\lambda^0 f  |  + \int\limits_0^1 \left| \frac{d}{d \e} B_\lambda^\tau f     \right| d\e,
$$
from which, using Minkowski's inequality for integrals we get
\begin{equation}\label{B-lambda-star}
 || B_\lambda^* ||  \lesssim || B_\lambda^0 ||  + 
 \int\limits_0^1 \left| \left|  \frac{d}{d \e} B_\lambda^\e  \right| \right|  d \e.
\end{equation}
Rewriting $B_\lambda^\e$ as volume integral gives
\begin{multline*}
B_\lambda^\e f(x,x_{n+1}) =  \int_{\R^n } e^{ i \lambda x \cdot y   }  
\varphi_0( (x,x_{n+1}), y ) K(x -y , x_{n+1} - \e^{\gamma} \psi(y  / \e) )
\times  \\  \varphi((x -y ) \lambda^\beta )  ( 1+| \e^{\gamma-1} (\nabla  \psi) ( y / \e ) |^2  )^{1/2}   f(y ) d y ,
\end{multline*}
from where we obtain 
\begin{equation}\label{B-lambda-0}
 B_\lambda^0 f(x,x_{n+1}) = \int_{ \R^n  }
 e^{ i \lambda x \cdot y   } \varphi_0( (x, x_{n+1}), y ) K(x-y, x_{n+1}   ) \varphi((x -y ) \lambda^\beta ) f(y ) d y .
\end{equation}
By (\ref{B-lambda-0}) and Lemma \ref{Lem-fixed-surface} we have
\begin{equation}\label{B-lambda-0-norm-est}
 || B_\lambda^0||  \lesssim \lambda^{ n(\beta - 1/2) - m\beta   }.
\end{equation}
For the derivative of $B_\lambda^\e$ with respect to $\e$ one has
$$
\frac{d}{d \e}  (B_\lambda^\e f(x,x_{n+1}) )  =    \mathcal{A}_\lambda^\e f( x,x_{n+1} ) + \mathcal{B}_\lambda^\e f( x,x_{n+1} ),
$$
where $\mathcal{A}_\lambda^\e$ contains the
derivatives of $( 1+| \e^{\gamma-1} (\nabla  \psi) ( y/ \e ) |^2  )^{1/2}$, and in $\mathcal{B}_\lambda^\e$
we collect the differential of the kernel. We have 
\begin{multline*}
\frac{d}{d \e} \left(  1+    \e^{2\gamma-2}  | (\nabla  \psi) (y/ \e) |^2    \right)^{1/2} = 
\left(  1+   \e^{2\gamma-2} |(\nabla  \psi) (y  / \e) |^2    \right)^{-1/2} \times \\
 \big[ (2\gamma -2) \e^{2\gamma -3}  | (\nabla  \psi) (y  / \e) |^2  - 
2 \e^{ 2\gamma -4  } (\nabla  \psi)^T (y  / \e)  (\mathrm{Hess} \psi) (y  / \e) y      \big].
\end{multline*}
Now if $\gamma>3/2$ by (\ref{B-lambda-norm}) from the last expression we obtain
\begin{equation}\label{A-lambda-eps-est}
 || \mathcal{A}_\lambda^\e   || \lesssim_\psi a(\e) \lambda^{ n(\beta -1/2) -m\beta +\beta },
\end{equation}
where $a(\e)$ is a positive and integrable function on the interval $(0,1)$, and the constant depends
on the bound of derivatives of $\psi$ up to second order. 
We will assume that $\beta< \frac{n}{2(n-m+1)}$ to get a decay in the norm of $\mathcal{A}_\lambda^\e$.

Next, we proceed to the estimate of $\mathcal{B}_\lambda^\e$
which contains the derivative of the kernel $K$. Differentiating the kernel we get
\begin{multline*}
\frac{d}{d \e} K( x  - y , x_{n+1} - \e^\gamma \psi(y  / \e)  ) = 
\partial_{n+1} K( x  - y , x_{n+1} - \e^\gamma \psi(y  / \e)  ) \times \\
[ -\gamma \e^{\gamma-1} \psi(y  / \e) + \e^{\gamma-2} \nabla  \psi(y  / \e) \cdot y  ].
\end{multline*}
Thus we get an operator as in (\ref{B-lambda-norm}) however the kernel here has higher singularity.
Applying the estimate (\ref{B-lambda-norm}) with $m$ replaced by $m-1$ implies
\begin{equation}\label{B-lambda-eps-est}
 || \mathcal{B}_\lambda^\e ||  \lesssim_\psi b(\e) \lambda^{ n(\beta -1/2) - m\beta + 2\beta },
\end{equation}
which gives a decay if 
\begin{equation}\label{beta-in-B-lambda-eps}
0<\beta<\frac{n }{2(n-m+2)}.
\end{equation}
Here as well, $b(\e)$ is a positive and integrable function on $(0,1)$.

Now combining estimates (\ref{A-lambda-eps-est}) and (\ref{B-lambda-eps-est}), along with the
estimate of $B_\lambda^0$ given by (\ref{B-lambda-0-norm-est}), from (\ref{B-lambda-star}) we obtain
\begin{equation}\label{B-star-final}
 || B_\lambda^* ||  \lesssim_\psi   \lambda^{ n(\beta -1/2) - m\beta + 2\beta  },
\end{equation}
which gives a decay if $0<\beta<1$ satisfies
\begin{equation}\label{beta-min}
 0< \beta <\frac{n}{ 2(n-m+2) }.
\end{equation}
Putting together Lemma \ref{Lem-A-lambda-max} and estimate (\ref{B-star-final}) we get
\begin{equation}\label{B-lambda-final-2}
 || T_\lambda^*||  \lesssim_\psi  \lambda^{ -m\beta  } + \lambda^{ n(\beta -1/2) - m\beta + 2 \beta   } ,
\end{equation}
which holds under the condition (\ref{beta-min}). Optimizing in $\beta$ gives $\beta= \frac{n}{2(n+2)}$ which
satisfies (\ref{beta-min}). For this choice of $\beta$ the proof of the Theorem is complete.
\end{proof}


\section{Nonlinear phase}\label{sec-nonlinear}

In this section we will consider one particular case of a nonlinear phase function with an application to Helmholtz equation.
Let $\psi_0 \in C_0^\infty(\R^n \times \R^{n-1} )$ and set $K(z) = |z|^{-(n-m-1)}$, $z\in \R^n\setminus\{0\}$,
where $0<m<n-1$, and $n\geq 2$. For $f \in L^2(\R^{n-1})$ consider the operator
\begin{equation}\label{def-of-T-Helmholtz}
T_\lambda f(x) = \int\limits_{\R^{n-1}} e^{i\lambda |x - (y',0)|^\gamma } \psi_0 (x, y') K \big( x- (y',0) \big) f(y') dy', 
\end{equation}
for $x\in \R^n $, $\gamma\geq 1$,  and $\lambda \geq 2$. We shall study decay of the norm of $T_\lambda$
as an operator from $L^2 (\R^{n-1})$ to $L^2(\R^n)$. In this section we denote this norm by $|| T_\lambda||$,
and have the following result.

\begin{theorem}\label{Thm-Helmholtz}
Set $\alpha= (n-1)/2$, and assume $\gamma>1$. Then
one has
$$ 
|| T_\lambda ||  \lesssim 
 \begin{cases}
                                 \lambda^{-\frac{1}{\gamma} (m+1/2)}, &\text{  $m<\gamma \alpha - 1/2$}, \\
                                 \lambda^{-\alpha} \log \lambda, &\text{  $m=\gamma \alpha - 1/2$}, \\
                                 \lambda^{-\alpha}, &\text{  $m>\gamma \alpha - 1/2$} .
                            \end{cases}
$$
\end{theorem}

\bigskip

\subsection{Preliminaries}

Before proving Theorem \ref{Thm-Helmholtz} we shall make some preliminary observations.
We assume $\gamma\geq 1$ and let
$$
\Phi(x',\xi) = | \xi - (x',0) |^\gamma = ( |\xi'  - x'|^2 + \xi_n^2 )^{\gamma /2} = d^\gamma,
$$
where $d =( |\xi'  - x'|^2 + \xi_n^2 )^{1/2}$, $1/2\leq d \leq 2$,  $\xi=(\xi', \xi_n)\in \R^{n-1}\times \R$,
and $x'=(x_1,...,x_{n-1})$. We have
$$
\frac{\partial^2 \Phi}{\partial x_i \partial \xi_j} = - \gamma d^{\gamma-2} \left( \delta_{ij} + 
(\gamma-2) \frac{x_i - \xi_i}{d}  \frac{x_j - \xi_j}{d}  \right),
$$
for $1\leq i,j \leq n-1$, where $\delta_{ij}$ is the Kronecker symbol. Setting
$a_i = (x_i - \xi_i)/d$ for $i=1,2,...,n-1$ we have
$$
\frac{\partial^2 \Phi}{\partial x_i \partial \xi_j} =  -\gamma d^{\gamma -2} \left( \delta_{ij} + (\gamma-2) a_i a_j \right).
$$
We set $D: = \mathrm{det} ( \delta_{ij} + (\gamma-2) a_i a_j )_{i,j=1}^{n-1}$.
The determinant here can be computed, for example, by Sylvester's determinant theorem (cf. ``matrix determinant lemma''), which states that
 $\mathrm{det} (I_k + AB )= \mathrm{det}(I_p + BA)$,
where matrices $A,B$ have dimensions respectively $k\times p$ and $p \times k$,
and $I_k$ and $I_p$ are correspondingly $k\times k$, and $p\times p$ identity matrices.
Applying this identity to $D$ we obtain
$D = 1+ (\gamma-2) \sum_{i=1}^{n-1} a_i^2$.
Thus
\begin{multline*}
D=1+(\gamma-2) \frac{|\xi' - x'|^2}{d^2} = \frac{|\xi' - x'|^2 + \xi_n^2 + (\gamma-2)|\xi' - x'|^2}{d^2} = \\
\frac{(\gamma-1) |\xi' - x'|^2 + \xi_n^2 }{d^2},
\end{multline*}
from which we conclude that
\begin{equation}\label{gamma-larger-1}
 \text{for } \gamma>1 \text{ one has } D\geq c>0 \text{ for }  \frac 12 \leq d \leq 2,
\end{equation}
and
\begin{equation}\label{gamma-is-1}
 \text{for } \gamma=1 \text{ one has } D \geq c>0 \text{ for } \frac 12 \leq d \leq 2 \text{ and } |\xi_n|\geq c_1 >0.
\end{equation}
For the proof of Theorem \ref{Thm-Helmholtz} we will use the following result.
\begin{theorem}\label{Thm-Stein}{\normalfont(see Stein \cite{Stein}, p. 377)}
Let $\psi_1 \in C_0^\infty(\R^n \times \R^n)$ and $\lambda>0$ and let $\Phi$ be real-valued and smooth.
Set
$$
\mathcal{U}_\lambda f(\xi) = \int_{\R^n} e^{i \lambda \Phi(x, \xi)} \psi_1(x,\xi) f(x) d x , \qquad \xi \in \R^n,
$$
and assume that $\mathrm{det} \left(  \frac{\partial^2 \Phi(x, \xi)}{ \partial x_i \partial \xi_j } \right) \neq 0$
on the support of $\psi_1$. Then one has
$$
|| \mathcal{U}_\lambda f ||_{L^2(\R^n) } \leq C \lambda^{-n/2} || f ||_{L^2(\R^n)}.
$$
\end{theorem}

\noindent We next give the proof of the main result of this section.
\bigskip

\subsection{Proof of Theorem \ref{Thm-Helmholtz}} 
 We start with a decomposition of the kernel $K$. Following Stein \cite{Stein} page 393,
there exists a function $\psi \in C_0^\infty (\R^n)$ such that
$\mathrm{supp} \psi \subset \{ x\in \R^n: \ \frac 12 \leq |x| \leq 2 \}$, and
$$
K(z) = \sum\limits_{k=-\infty}^\infty 2^{k (n-1-m)} \psi  (  2^k z  ).
$$
As such function one may take $\psi(z) = |z|^{-(n-1-m)} [ \eta( z ) - \eta( 2 z) ] $ where $\eta$
is a smooth function satisfying $\eta(y) =1 $ for $|y| \leq 1$ and $\eta (y) =0$ for $|y| \geq 2$.
It is clear that there is $k_0\in \Z$ such that
$$
K(z) = \sum\limits_{k=k_0}^\infty 2^{k (n-1-m)} \psi (  2^k z ), \text{ for } z=x-(y',0) \text{ where } (x,y')\in \mathrm{supp} \psi_0.
$$
We will assume that $k_0=0$, the proof in the general case is the same. We get
$$
T_\lambda f= \sum\limits_{k=0}^\infty T_{\lambda, k } f ,
$$
where  
$$
T_{\lambda, k} f(x) = \int\limits_{\R^{n-1}} e^{ i \lambda | x- (y',0) |^\gamma } \psi_0(x, y') 2^{ k(n-1-m) } 
\psi \big( 2^k ( x-  (y',0) ) \big) f(y') dy' .
$$
From an application of Hardy-Littlewood's maximal operator (see the proof of Lemma \ref{Lem-A-lambda})
we have $ || T_{\lambda, k}  || = || T_{\lambda, k} ||_{L^2(\R^{n-1}) \mapsto L^2(\R^n)}  \leq C 2^{-mk}$.
Here we shall improve this estimate.

Take $\varphi \in C_0^\infty (\R) $ such that $\mathrm{supp} \varphi \subset [ -\frac 12 - \frac{1}{10},  \frac 12 + \frac{1}{10} ]$
and $\sum\limits_{j=-\infty }^\infty \varphi( t-j )=1$. Set $\varphi_j (t) = \varphi (t-j)$, and for $t=(t_1,...,t_{n-1})$ define
$ \chi ( t ) := \varphi  (t_1 ) \varphi  (t_2  )\cdot ... \cdot  \varphi (t_{n-1}) $.
Then, for $j=(j_1,...,j_{n-1})\in \Z^{n-1}$ set
$$
\chi_j( t ) := \chi(t-j) =   \varphi_{j_1} (t_1 ) \varphi_{j_2} (t_2  )\cdot ... \cdot  \varphi_{j_{n-1}} (t_{n-1}).
$$
Clearly $\sum\limits_{j\in \Z^{n-1}} \chi_j(t)=1$, and
$$
1 = \sum\limits_{j\in \Z^{n-1}} \chi_j (2^k t) = \sum\limits_{j\in \Z^{n-1}} \chi ( 2^k t -j  ) = 
\sum\limits_{j\in \Z^{n-1}} \chi \big( 2^k ( t -2^{-k} j )   \big).
$$
Setting $f_j (t): = f( t ) \chi \big( 2^k ( t-2^{-k} j ) \big) $, implies
$$
f = \sum\limits_{ j\in \Z^{n-1} } f_j.
$$
We get the following decomposition
$$
T_{\lambda, k} f (x) = \sum\limits_{ j\in \Z^{n-1} } T_{\lambda, k} f_j (x),
$$
and for every $x$ the sum has only a bounded number of non-vanishing terms. Therefore by Cauchy-Schwarz we obtain
\begin{equation}\label{Cauchy-Schwarz}
 | T_{\lambda,k} f(x)  |^2 \leq C \sum\limits_{ j\in \Z^{n-1} } | T_{\lambda, k} f_j (x) |^2,
\end{equation}
where the constant $C$ is independent of $x$, $\lambda$, and $k$. We have 
\begin{multline*}
T_{ \lambda, k } f_j (x) = \int\limits_{ \R^{n-1} } e^{ i\lambda |x - (y',0) |^\gamma  } \psi_0 (x, y') 2^{ k(n-1-m) }
\psi \big( 2^k ( x- (y', 0) ) \big) f_j (y') dy' = \\ \big[\text{with } y'=2^{-k} z' \big] \ 
2^{-m k } \int\limits_{ \R^{n-1} } e^{ i \lambda | x- 2^{-k} (z', 0) |^\gamma } \psi_0 (x, 2^{-k} z') \psi \big( 2^k x- (z',0) \big) f_j (2^{-k} z') dz',
\end{multline*}
and
$$
f_j( 2^{-k} z' ) = f(  2^{-k} z') \chi \big(  2^k ( 2^{-k} z' - 2^{-k} j ) \big) = f(  2^{-k} z') \chi ( z' - j ).
$$
Hence
\begin{multline*}
 T_{ \lambda, k } f_j (x) = 2^{- m k} \int\limits_{ \R^{n-1} } e^{ i \lambda 2^{-k \gamma} | 2^k x - (z',0) |^\gamma }
 \psi_0 (x, 2^{-k} z' ) \psi \big( 2^k x - (z',0) \big) f(2^{-k} z' ) \chi( z' -j ) dz' = \\
 \big[ \text{with } y'= z' -j \big]  \ 
 2^{-m k} \int\limits_{ \R^{n-1} } e^{ i\lambda 2^{-k \gamma} | 2^k x - (y' +j,0) |^\gamma  } \psi_0 \big( x, 2^{-k} (y'+j) \big)
 \psi \big( 2^k x- (y'+j,0) \big) \times \\ f \big( 2^{-k} (y'+j) \big) \chi(y') dy' =  2^{-m k } 
\int\limits_{ \R^{n-1} } e^{ i \lambda 2^{-k \gamma} \left| 2^k ( x- (2^{-k}j,0) ) - (y',0) \right|^\gamma  } \psi_0 (x, 2^{-k } j+ 2^{-k} y')
\times \\
 \psi \big(2^k (x- (2^{-k} j,0) ) -(y',0) \big) f(  2^{-k} j + 2^{-k} y' ) \chi(y') d y'.
\end{multline*}
We also have
\begin{multline*}
 \int\limits_{ \R^n } | T_{\lambda,k} f_j (x)  |^2 dx  = \big[ \text{with } x= u+(2^{-k} j, 0) \big] \\
 \int\limits_{ \R^n } \left| T_{\lambda,k} f_j \big( u+ (2^{-k} j, 0 ) \big) \right|^2 d u =  \big[ \text{with } \xi=  2^k u \big] \\
 2^{-k n} \int\limits_{\R^n }  \left| T_{\lambda,k} f_j \big( 2^{-k} \xi + (2^{-k}j, 0) \big)  \right|^2 d \xi .
\end{multline*}
Now let $\widetilde{\chi} \in C_0^\infty (\R^{n-1})$ be so that $\widetilde{\chi} =1$ on $\mathrm{supp} \chi$
and $\mathrm{supp} \widetilde{\chi} \subset [-1,1]^{n-1}$. We then have
\begin{multline*}
 T_{\lambda, k } f_j \big( 2^{-k} \xi + ( 2^{-k}j,0 )  \big) = 
 2^{- m k} \int\limits_{ \R^{n-1} } e^{ i \lambda 2^{-k \gamma} | \xi - (y',0) |^\gamma } \psi_0( 2^{-k} \xi +(2^{-k}j, 0 ), 2^{-k} j + 2^{-k} y' ) \times
 \\ \psi \big( \xi - (y',0) \big)  f( 2^{-k} j + 2^{-k} y' ) \chi(y') \widetilde{\chi}(y') dy' = 
 2^{-m k } \int\limits_{ \R^{n-1} } e^{ i\lambda 2^{-k \gamma} \Phi (y', \xi) } \psi_1 (y', \xi) g(y') dy' := \\ 2^{- m k} \mathcal{U}_{ \lambda 2^{-k \gamma} } g(\xi),
\end{multline*}
where
\begin{equation}\label{Phi-in-thm}
\Phi (y' ,\xi) = | \xi - (y', 0) |^\gamma = \big( |  \xi' - y'|^2 + \xi_n^2 \big)^{\gamma/2},
\end{equation}
$$
\psi_1(y', \xi)  = \psi \big(\xi - (y',0) \big) \psi_0 \big(  2^{-k} \xi + (2^{-k} j, 0), 2^{-k}  j + 2^{-k} y' \big) \widetilde{\chi}(y'),
$$
and
$$
g(y') = f( 2^{-k} j + 2^{-k} y') \chi(y').
$$
It is clear that for $R>0$ large enough, independently of $j$ and $k$, one has
$$
\mathrm{supp} \psi_1 \subset B(0,R) \times B(0,R),
$$
and that the derivatives of $\psi_1$ can be bounded uniformly in $k$ and $j$.
We have
$$
\left( \int_{\R^n} |T_{\lambda, k} f_j(x)|^2  dx \right)^{1/2}  = 2^{-kn/2} 2^{-mk} \left( \int_{\R^n} | \mathcal{U}_{\lambda 2^{-k \gamma}} g(\xi) |^2 d \xi \right)^{1/2} 
$$
and we shall now use Theorem \ref{Thm-Stein}.
It follows from (\ref{gamma-larger-1}) that the determinant condition in Theorem \ref{Thm-Stein}
is satisfied and hence we obtain
$$
\left( \int_{\R^{n-1}} | \mathcal{U}_{\lambda 2^{-k \gamma} } g( \xi', \xi_n ) |^2 d \xi' \right)^{1/2} \leq C (\lambda 2^{-k \gamma})^{-\alpha}
\left( \int_{\R^{n-1}} | g(x') |^2 dx'  \right)^{1/2},
$$
where $\alpha = (n-1)/2$. Integration in $\xi_n$ demonstrates
$$
|| \mathcal{U}_{\lambda 2^{-k \gamma}} g ||_{L^2(\R^n)} \leq C \lambda^{-\alpha} 2^{k \gamma \alpha}  || g ||_{L^2(\R^{n-1})}.
$$
We also have
\begin{multline*}
|| g ||_{L^2(\R^{n-1})}^2 = \int\limits_{ |y'|\leq \sqrt{n} }  | f( 2^{-k} j + 2^{-k} y' ) |^2 dy' = \ [ \text{with } z' = 2^{-k} y' ] \\
2^{ k(n-1) } \int\limits_{ |z'| \leq 2^{-k} \sqrt{n} } | f( 2^{-k} j + z'  ) |^2 dz' .
\end{multline*}
Here $ g= g_{j,k} $ and it is easy to see that
$$
\sum\limits_{j\in \Z^{n-1}} \int_{\R^{n-1}} | g_{j,k}  (y') |^2 d y'  \leq C 2^{ k(n-1) } || f ||_{L^2(\R^{n-1})}^2.
$$
Invoking (\ref{Cauchy-Schwarz}) we obtain
\begin{multline*}
\int_{\R^n} | T_{\lambda,k} f (x) | ^2 dx  \leq C \sum\limits_{j \in \Z^{n-1}} \int_{\R^n} | T_{\lambda, k} f_j  |^2 dx \leq \\
C 2^{-k n } 2^{-2 m k} ( \lambda^{-\alpha} 2^{k \gamma \alpha} )^2 \sum\limits_{j } \int_{\R^{n-1}} |g_{j,k}|^2 dy' \leq \\
C 2^{-k n } 2^{-2 m k} ( \lambda^{-\alpha} 2^{k \gamma \alpha} )^2 2^{ k(n-1) } \int_{\R^{n-1}} |f |^2 dy' 
\end{multline*}
and hence
\begin{equation}\label{T-lambda-k-est}
|| T_{\lambda, k} ||  \leq C 2^{-k/2} 2^{- mk} \lambda^{-\alpha} 2^{k \gamma \alpha}.
\end{equation}
In this estimate $\lambda^{-\alpha} 2^{k \gamma \alpha}$ can be replaced by 1, since we can make a trivial
estimate instead of using Theorem \ref{Thm-Stein}. Thus we also have
\begin{equation}\label{T-lambda-k-est-trivial}
|| T_{\lambda, k} ||  \leq C 2^{-k/2} 2^{- mk} .
\end{equation}
It follows that
$$
|| T_\lambda || \leq \sum\limits_{k=0}^\infty || T_{\lambda, k} || \lesssim \lambda^{-\alpha} 
\sum_{2^k \leq \lambda^{1/\gamma}} 2^{ (\gamma \alpha - m - 1/2 ) k}  + \sum_{ 2^k  \geq \lambda^{1/\gamma}  } 2^{ - (m + 1/2) k }.
$$
For $m< \gamma \alpha -1/2$ we have $\gamma \alpha - m - 1/2>0$ and
$$
|| T_\lambda || \lesssim \lambda^{-\alpha} \lambda^{( \gamma \alpha - m - 1/2 ) / \gamma} + \lambda^{-(m + 1/2 ) / \gamma} \leq C \lambda^{-(m + 1/2 ) / \gamma}.
$$
If $m = \gamma \alpha -1/2$ one has $\gamma \alpha - m - 1/2 =0$ and
$$
|| T_\lambda || \lesssim \lambda^{-\alpha} \log \lambda + \lambda^{-\alpha} \leq C \lambda^{-\alpha} \log \lambda .
$$
Finally, for $m > \gamma \alpha -1/2$ we have $ \gamma \alpha -m -1/2<0$ and
$$
|| T_\lambda || \lesssim  \lambda^{-\alpha} + \lambda^{-(m+1/2)/\gamma}  \leq C \lambda^{-\alpha},
$$
and the proof of the Theorem is complete. \qquad \qquad \qquad \qquad \qquad \qquad \qquad \qquad \qquad $\square$

\subsubsection{\normalfont{\textbf{The case of $\gamma=1$, and the Helmholtz equation.}}}\label{sub-sec-Helmholtz}
Here we set $K(z) = |z|^{-(n-1-m)} \omega(z)$,
where $\omega \in C^\infty(\R^n \setminus \{0\})$, is homogeneous of degree 0,
and $\omega(z) = 0 $ for all $|z|=1$ satisfying $|z_n| \leq \e$, for some given $\e>0$. 
The method of the proof of Theorem \ref{Thm-Helmholtz} combined with 
(\ref{gamma-is-1}) gives
\begin{equation}\label{Helmholtz-est-gamma-1} 
|| T_\lambda ||  \lesssim 
 \begin{cases}
                                 \lambda^{-  (m+1/2)}, &\text{  $m<  n/2 - 1$}, \\
                                 \lambda^{-\alpha} \log \lambda, &\text{  $m=n/2 -1$}, \\
                                 \lambda^{-\alpha}, &\text{  $m> n/2 - 1 $} ,
                            \end{cases}
\end{equation}
with $\alpha=(n-1)/2$.

We now discuss how the analysis can be applied to Helmholtz equation.
In $\R^3$ consider a smooth hypersurface $\Gamma$ given by (\ref{Gamma-surface}),
and fix some $\varphi_0\in C_0^\infty(\R^3)$. Let $\sigma_\Gamma$ be the surface measure
of $\Gamma$, and for a measure $\mu_\Gamma  = \varphi_0(y) d \sigma_\Gamma(y)$ let 
$u_\lambda$ be a solution to the following inhomogeneous Helmholtz equation
\begin{equation}\label{eq-Helmholtz}
\Delta u + \lambda^2 u =  -\mu_\Gamma   \ \ \text{ in } \R^3.
\end{equation}
We denote by $G_\lambda (\cdot, \cdot)$ the fundamental solution (Green's function) of (\ref{eq-Helmholtz}). It is well known that 
$G_\lambda (x,y)=e^{i\lambda |x-y|}/ (4\pi |x-y|)$ where $x,y\in \R^3$, $x\neq y$. Now let $u_\lambda$ be the solution
to (\ref{eq-Helmholtz}) given in terms of the Green's function, namely
\begin{equation}\label{Helm-sol-Green}
u_\lambda (x) = \int_{\R^3} G_\lambda(x,y) d\mu_\Gamma (y).
\end{equation}
This is precisely the solution to (\ref{eq-Helmholtz}) satisfying Sommerfeld radiation condition, which in dimension three reads
$$
 \lim\limits_{r \to \infty } r \left( \frac{\partial}{\partial r } -  i \lambda \right) u_\lambda (r \nu) =0,
$$
uniformly for all directions $\nu \in \mathbb{S}^2$, where $i$ is the imaginary unit.
From (\ref{Helm-sol-Green}) we have
\begin{equation}\label{Helm-3d-on-Gamma}
u_\lambda (x) = \int_{\Gamma}  \frac{e^{i \lambda |x-y|}}{4\pi |x-y|}  \varphi_0(y) d\sigma_\Gamma(y), \qquad x\in \R^3.
\end{equation}
Let us show here how to obtain decay estimates on $u_\lambda$, as $\lambda \to \infty$, in the case when $\Gamma$ is a plane.
Assume $ \nu \in \R^3$ is the unit normal to $\Gamma$, thus $\Gamma= \{ y \in \R^3: y\cdot \nu =0 \}$.
Fix any $3 \times 3$ orthogonal matrix $M$ such that $M e_3 = \nu$, where $e_3=(0,0,1)$.
Next, make a change of variables in (\ref{Helm-3d-on-Gamma}) by the formula $y=M z$, and $x=M w$, where $w, z \in \R^3$. We get
\begin{multline}\label{Helm-reduced-to-0-plane}
u_\lambda (x) =u_\lambda (Mw) =  \int_{ z_3=0 }  \frac{e^{i \lambda |Mw-Mz|}}{4\pi |M w- Mz|}  \varphi_0(M z) d\sigma (z) = \\ \int_{\R^2} 
\frac{e^{i \lambda | w- (z',0)  |}}{4\pi | w- (z',0)|}  \varphi_0(M (z',0)) d z', \qquad x,w \in \R^3.
\end{multline}

Thus the case of an arbitrary plane, by a rotation, is easily reduced to the case of $y_3=0$ in $\R^3$.
Now, if we have a bounded domain $D\subset \R^3$
which stays within a positive distance from $\Gamma$, then the determinant condition (\ref{gamma-is-1}) will be satisfied for all $x\in D$.
Next, using the argument of Theorem \ref{Thm-Helmholtz} as we do for (\ref{Helmholtz-est-gamma-1}),
with $n=3$ and $m=1$ one may show from (\ref{Helm-reduced-to-0-plane}) that 
\begin{equation}
|| u_\lambda ||_{L^2(D)} \lesssim_D \lambda^{- 1} |D|^{1/2}  , 
\end{equation}
where the constant in the inequality depends on the distance
of $D$ and the plane $\Gamma$.
Thus, in this particular case
we get a quantitative decay estimate of the $L^2$-norm of solutions $u_\lambda$ as the eigenvalue tends to infinity.

When we allow $D$ to cross the plane $\Gamma$, the determinant condition in (\ref{gamma-is-1}) becomes
invalid. However, this scenario can be handled by splitting the integral in (\ref{Helm-reduced-to-0-plane})
by means of a smooth cut-off into two regions, where one of them stays away from $\Gamma$, and the other one is in a small neighbourhood of $\Gamma$.
The former will be estimated as above, relying on (\ref{gamma-is-1}), while the latter will be handled using the smallness of the support
of the cut-off and uniform boundedness of $u_\lambda$. Thus, one may obtain decay estimates on $u_\lambda$ in this setting
as well relying on the methods discussed in the paper, however, the rate of decay will be worse than $\lambda^{-1}$. 
In a similar vein we may handle the case when $\Gamma$ is a union of some finite number of planes in $\R^3$,
by applying the analysis on each flat piece of $\Gamma$ separately.
The details are left as an exercise for an interested reader.

It should be noted that at this stage our analysis does not extend to the case of a general hypersurface
and more interestingly to the case of oscillating (rough) boundaries considered in subsection \ref{sec-perturb-surface} and Section \ref{sec-maximal-operator}.
Understanding the properties of the operator (\ref{def-of-T-Helmholtz}) in this generality seems to be a very interesting problem.
Another interesting problem here is to understand if the operators 
$T_\lambda$ with non integer $\gamma$ can be applied to the study of fractional order Helmholtz operators.
It was to our surprise that there were virtually no results in the literature concerning fundamental solutions
of fractional order Helmholtz equations.

\subsubsection{\normalfont{\textbf{The case of $\gamma=2$.}}} We have $\gamma \alpha -1/2 = n-1-1/2 = n-3/2$. For $m<n-3/2$
we have proved that 
\begin{equation}\label{T-lambda-est-gamma-2}
 || T_\lambda ||_{L^2(\R^{n-1}) \mapsto L^2(\R^n)} \leq C \lambda^{-(m/2 + 1/4)}.
\end{equation}
Now set
$$
S_\lambda f(x) = \int_{\R^{n-1}} e^{i \lambda x' \cdot y'} \psi_0(x, y') K( x-(y',0) )  f(y') dy'.
$$
We have
$$
2 x' \cdot y' = 2 x \cdot (y',0) = |x|^2 + |y'|^2  - |x-(y',0)|^2
$$
and
$$
e^{i 2 \lambda x' \cdot y'} = e^{i \lambda |x|^2}  e^{i \lambda |y'|^2} e^{i \lambda |x-(y',0)|^2}.
$$
Combination of this inequality and (\ref{T-lambda-est-gamma-2}) gives
\begin{equation}\label{S-lambda-est-gamma-2-n}
 ||S_\lambda ||_{L^2(\R^{n-1}) \mapsto L^2(\R^n)} \leq C \lambda^{-(m/2+1/4)}.
\end{equation}
On the other hand taking $\beta=1/2$ in Lemmas \ref{Lem-A-lambda} and \ref{Lem-fixed-surface} we obtain
\begin{equation}\label{S-lambda-est-gamma-2-n-1}
 ||S_\lambda ||_{L^2(\R^{n-1}) \mapsto L^2(\R^{n-1})} \leq C \lambda^{-m/2 }.
\end{equation}
It is interesting to observe that we have a better decay in (\ref{S-lambda-est-gamma-2-n}) than in (\ref{S-lambda-est-gamma-2-n-1}).

\subsection{Lower bounds.}\label{subsec-lower-bounds}

Here we construct examples for which the upper bounds
(\ref{S-lambda-est-gamma-2-n}), (\ref{S-lambda-est-gamma-2-n-1}), and the first estimate of Theorem \ref{Thm-Helmholtz}
are attained, showing that these estimates are best possible.

We start with an example showing that $m/2$ in (\ref{S-lambda-est-gamma-2-n-1}) can not be replaced by a larger number.
Take $x_0' \in \R^{n-1}$  and set $F = \{ x' \in \R^{n-1}: \ |x' - x_0'| \leq c_0 \lambda^{-1/2} \}$
where $c_0$ is a small constant, which will be specified in due course. Then choose $y_0' \in \R^{n-1}$ such that $|x_0' - y_0' | = 100 c_0 \lambda^{-1/2}$
and set $E= \{ y' \in \R^{n-1}: \ |y' - y'_0| \leq c_0 \lambda^{-1/2} \}$. We have
$$
x' \cdot y' = (x' - x_0') \cdot (y' - y_0') + (x' - x_0')\cdot y_0'+ x_0' \cdot y'
$$
and
$$
e^{ i \lambda x' \cdot y' } = e^{i \lambda (x'- x_0') \cdot (y'-y_0')} e^{i \lambda (x'- x_0') \cdot y_0'} e^{i \lambda x_0' \cdot y'}.
$$
Setting $f(y') = e^{-i \lambda x_0' \cdot y'} \mathbb{I}_E (y')$ we obtain
\begin{align*}
S_\lambda f(x) = e^{i \lambda (x' - x_0')\cdot y_0'} \int_{\R^{n-1}} e^{i \lambda (x' - x_0') \cdot (y'-y_0')} \psi_0 (x,y')
e^{i\lambda x_0' \cdot y'} K(x-(y',0)) f(y') dy' &= \\
e^{i \lambda (x' - x_0')\cdot y_0'} \int_{\R^{n-1}} \psi_0(x,y') e^{i \lambda (x'-x_0') \cdot (y' - y_0')} K(x-(y',0)) \mathbb{I}_E(y') dy' &= \\
e^{i \lambda (x' - x_0')\cdot y_0'}  \int_E \psi_0(x,y') K(x-(y',0)) dy' &+ \\
e^{i \lambda (x' - x_0')\cdot y_0'} \int_E \psi_0(x,y') \big( e^{i\lambda (x' - x_0')\cdot (y'-y_0')} - 1 \big) K(x-(y',0)) dy' &:= \\ P(x) + R(x),
\end{align*}
where we have
\begin{multline*}
 |R(x)| \leq C \int_E \lambda |x'- x_0'| |y'-y'_0| K(x-(y',0)) dy' \leq \\ C \int_E \lambda c_0 \lambda^{-1/2} c_0 \lambda^{-1/2} K(x-(y',0)) dy'=
C c_0^2 \int_E K(x-(y',0)) dy',
\end{multline*}
for $x'\in F$. We now take $x=(x',0)$ and assume that $\psi_0( (x', 0),y') = 1$ for $x' \in F$ and
$y' \in E$. We then have
$$
|P(x)| = \int_E K(x'-y',0) dy',
$$
and
$$
|R(x)| \leq C c_0^2 \int_E K(x' - y',0) dy', \qquad \text{ for } x' \in F.
$$
Choosing $c_0$ small, for $x' \in F$ we obtain
$$
|S_\lambda f(x',0) | \geq \frac 12 \int_E \frac{1}{|x'-y'|^{n-1-m} } dy' \geq c (\lambda^{-1/2}  )^{n-1} \frac{1}{\lambda^{-(n-1-m)/2}} = c \lambda^{-m/2}.
$$
It follows that
$$
\left( \int_{R^{n-1}}  |S_\lambda f(x',0) |^2 d x' \right)^{1/2} \geq c \lambda^{-m/2} |F|^{1/2}.
$$
But $||f||_2 = |E|^{1/2} = |F|^{1/2}$ and hence
$$
||S_\lambda||_{L^2(\R^{n-1}) \to L^2(\R^{n-1}) } \geq c \lambda^{-m/2}.
$$
Thus we have proved that (\ref{S-lambda-est-gamma-2-n-1}) is sharp.

\bigskip
We shall then prove that (\ref{S-lambda-est-gamma-2-n}) is sharp as well, for which we will
use the construction described above. Define $F_1 : = F \times [0, c_0 \lambda^{-1/2}]$
and assume that $\psi_0 (x, y') = 1$ for $x \in F_1$ and $y' \in E$. Now for $x \in F_1$
we obtain as above that $S_\lambda f(x) = P(x) + R(x)$ where
$$
|P(x)| =\int_E K(x' - y', x_n) dy'.
$$
It follows that for $x \in F_1$ we have
$$
|S_\lambda f(x) | \geq \frac 12 \int_E \frac{1}{\lambda^{-(n-1-m)/2}} dy' \geq c \lambda^{-m/2},
$$
hence
$$
\left(  \int_{\R^n} |S_\lambda f(x)|^2 dx  \right)^{1/2} \geq c \lambda^{-m/2} \lambda^{-n/4},
$$
since $|F_1| = c \lambda^{-n/2}$. But
$$
|| f ||_2 = |E|^{1/2} = c \lambda^{-(n-1)/4}= c \lambda^{-n/4 + 1/4},
$$
and consequently
$$
|| S_\lambda ||_2 \geq c ||f||_2 \lambda^{-m/2 - 1/4},
$$
which shows that 
$$
|| S_\lambda ||_{L^2 (\R^{n-1}) \mapsto L^2(\R^n) }  \geq c \lambda^{-(m/2+1/4)}.
$$
Thus we have proved that (\ref{S-lambda-est-gamma-2-n}) is sharp.

\bigskip

Finally we turn to the first estimate of Theorem \ref{Thm-Helmholtz}.
For $\gamma>1$ we have
$$
T_\lambda f(x) = \int_{\R^{n-1}} e^{ i \lambda | x- (y',0) |^\gamma } \psi_0 (x,y') K(x- (y',0)) f(y') dy'
$$
and we shall estimate the norm of $T_\lambda$ from below. Set $F := B(x_0', c_0 \lambda^{-\beta} )$,
$F_1: = F \times [0, c_0 \lambda^{-\beta}]$, and $E:=  B(y_0', c_0 \lambda^{-\beta})$, where
$|x_0' - y_0'| = 100 c_0 \lambda^{-\beta}$. Here $B(x,R)$ denotes a ball with a center $x$
and radius $R$. We assume that $\psi_0 (x,y') =1 $ for $x\in F_1$ and $y' \in E$.
Letting $x\in F_1$ and $f=\mathbb{I}_E$ we have
$$
T_\lambda f(x) = \int_E K(x-(y',0)) dy' + \int_E \big(  e^{i \lambda |x-(y',0)|^\gamma} -1 \big) K( x-(y',0) ) dy' := P(x) + R(x).
$$
Choosing $\beta=1/\gamma$ for $x \in F_1$ and $y' \in E$ implies
$$
|e^{i \lambda |x-(y',0)|^\gamma} -1 | \leq \lambda  |x-(y',0)|^\gamma  \leq C c_0 \lambda \lambda^{-\beta \gamma} \leq C c_0.
$$
It follows that
$$
|R(x)| \leq C c_0 \int_E K( x-(y',0) ) dy'.
$$
Now taking $c_0$ small we obtain
$$
| T_\lambda f(x) | \geq c \int_E K(x-(y',0)) dy' \geq c \int_E \frac{d y'}{\lambda^{-\beta(n-1-m)}} = C \lambda^{-\beta m}
$$
and
$$
\left( \int_{F_1} | T_\lambda f  |^2 dx \right)^{1/2} \geq c \lambda^{-\beta m} \lambda^{-\beta n/2},
$$
since $|F_1| \geq c \lambda^{-\beta n}$. But $ || f ||_2 = c \lambda^{-\beta(n-1)/2}  = c \lambda^{-\beta n/2}  \lambda^{\beta /2}$  
and hence
$$
|| T_\lambda|| \geq c \lambda^{-\beta m} \lambda^{-\beta /2} = c \lambda^{-\beta (m+1/2)}= c \lambda^{-( m+1/2) /\gamma}.
$$
In Theorem \ref{Thm-Helmholtz} we proved that
$$
|| T_\lambda|| \leq C \lambda^{-(m+1/2)/\gamma},
$$
for $m<\gamma \alpha - 1/2= \gamma n/2 - \gamma/2 - 1/2$
and thus we have proved that this estimate is sharp.

\section*{Acknowledgements}
Part of the paper has been written while all authors met at Institut Mittag-Leffler
during the program ``Homogenization and Random Phenomenon". H.A. thanks The Institue for its support. H.S. was partially supported by the Swedish Research Council.

\end{document}